\documentclass[leqno]{aadmbook}

\usepackage{amsmath,amsthm,amssymb}
\usepackage{scalerel,stackengine}
\usepackage{newtxtext,newtxmath}
\usepackage{graphicx}
\usepackage{amsfonts,amsthm}
\usepackage{epsfig}
\usepackage{mathrsfs}
\usepackage{lscape}
\usepackage{graphicx}
\usepackage{pgfplots}
\usepackage{enumitem}
\usepackage{dsfont}
\usepackage{mathtools}
\usepackage{amstext}
\usepackage{float} 

\newcommand{\dT}{{^{\rm{\mathbf{H}}}}{\mathfrak{D}}_{T}^{\alpha,\beta;\chi}}
\newcommand{\dTo}{{^{\rm{\mathbf{H}}}}{\mathfrak{D}}_{0+}^{\alpha,\beta;\chi}}

\newcommand{\R}{\mathbb{R}}
\newtheorem{theorem}{Theorem}

\newtheorem{lemma}[theorem]{Lemma}
\newtheorem{proposition}[theorem]{Proposition}
\newtheorem{corollary}[theorem]{Corollary}
\newtheorem{definition}[theorem]{Definition}

\begin{document}


\oddsidemargin 16.5mm
\evensidemargin 16.5mm

\thispagestyle{plain}

\begin{center}
{\large \sc  Applicable Analysis and Discrete Mathematics}

{\small available online at  http:/$\!$/pefmath.etf.rs }
\end{center}

\noindent{\small{\sc  Appl. Anal. Discrete Math.\ }{\bf x} (xxxx),xxx--xxx.}\\
\noindent{\scriptsize https://doi.org/10.2298/AADMxxxxxxxx}     

\vspace{5cc}
\begin{center}

{\large\bf  Existence and multiplicity for fractional Dirichlet problem with $\gamma(\xi)$-Laplacian equation and Nehari manifold
\rule{0mm}{6mm}\renewcommand{\thefootnote}{}
\footnotetext{\scriptsize ${}^{\ast}$Corresponding author. J. Vanterler da C. Sousa}
\footnotetext{

\rule{2.4mm}{0mm}Keywords and Phrases. Fractional differential equation. $\gamma(\xi)$-Laplacian equation. variational methods. Nehari manifold. multiple positive solutions.}}

\vspace{1cc}
{\large\it J. Vanterler da C. Sousa ${}^{\ast}$, D. S. Oliveira and Ravi P. Agarwal}

\vspace{1cc}
\parbox{24cc}{{\small This paper is divided in two parts. In the first part, we prove coercivity results and minimization of the Euler energy functional. In the second part, we focus on the existence and multiplicity of a positive solution of fractional Dirichlet problem involving the $\gamma(\xi)$-Laplacian equation with non-negative weight functions in $\mathcal{H}^{\alpha,\beta;\chi}_{\gamma(\xi)}(\Lambda,\mathbb{R})$ using some variational techniques and Nehari manifold.
}}

\end{center}


\vspace{1.5cc}
\section{Introduction and Motivation} 

Problem of variable exponent spaces $L^{p(x)}$ and the space $W^{1,p(x)}$ have been a subject of active research area \cite{Alves,Chabrowski,Diening,Emunds1,Edmunds2,Fan,Fan1,Mashiyev}. The specific attention accorded to such problems is due to their applications in mathematical physics. What has been noticed is a growing interest in elliptic problems in Sobolev space $W^{1,p(x)}$ using classical variational techniques. Researchers such as Radulescu \cite{Xiang}, Alves \cite{Alves}, Fan \cite{Fan}, Rabinowitz \cite{Rabinowitz}, Ambrosetti \cite{Ambrosetti}, Winkert \cite{Winkert}, Pucci \cite{Piersanti}, Motreanu \cite{Motreanu}, Papageorgiou \cite{Papa}, Bisci \cite{Figueiredo}, Repovs \cite{Repov}, among other researchers, have dedicated themselves to investigating cutting-edge problems using operators $p(x)$- Laplacian and performing applications.

In 2006 Mihailescu \cite{Miha} discuss the existence of solutions for the problem
\begin{eqnarray*}
\begin{dcases}
- div(|\nabla u|^{p(x)-2} \nabla u )=f(x,u) \quad {\rm in} \,\,\Lambda\\
u(x)=0 \quad {\rm on}\,\, \partial\Lambda.
\end{dcases}
\end{eqnarray*}
For more details see \cite{Miha}. Another interesting work on the existence of solutions involving $p(x)$-Laplacian was investigated by Alves and Barreiro \cite{Alves}. In 2015, Chabrowski and Fu \cite{Chabrowski}, considered the existence of solutions in $W^{1,p(x)}_{0}(\Lambda)$ for the $p(x)$-Laplacian problems in the superlinear and sublinear cases using the mountain pass theorem technique.

In 2007 Wu \cite{Wu} investigated the multiplicity of solutions using Nehari manifold for the elliptic equation 
\begin{eqnarray}
\begin{dcases}
-\Delta_{p} u=\lambda f(x) |u|^{q-2} u+ g(x)|u|^{r-2} u \quad {\rm in} \,\,\Lambda\\
u(x)=0 \quad {\rm on}\,\, \partial\Lambda
\end{dcases}
\end{eqnarray}
where $1<q<p<r<p^{*}$, $\Lambda\subset \mathbb{R}^{N}$ is a bounded domain, $\lambda\in \mathbb{R}/\left\{0 \right\}$, and the weight functions $f,g\in C(\overline{\Lambda})$ are satisfying $f^{\pm}=max\left\{\pm f,0 \right\} \neq 0$ and $g^{\pm}=max\left\{\pm g,0 \right\} \neq 0$. For more details, see \cite{Wu}.

On the other hand, in the recent years increasing attention has been paid to the study of fractional differential equations \cite{Diethelm,Kilbas,Lakshmikantham,Zhou}. Such equations are used to model phenomena in medicine, physics, engineering, biology, among other areas (see for instance \cite{almeirda,Diethelm,Kilbas,Lakshmikantham,ze,Zhou}  and the references therein). Recently, fractional differential equation problems involving $p$-Laplacian have gained attention from some researchers, in particular, involving the $\psi$-Hilfer fractional operator \cite{Biswas,Ledesma,Saoudi,torresn,Truong,Sousa10,Sousa11}. 

In 2020, Sousa et al. \cite{Sousa} proposed a work on variational problems using fractional derivatives. In this sense, the authors discuss the existence and nonexistence of weak solutions for the fractional $p$-Laplacian using the Nehari manifold and application of fibration, of the following problem
\begin{equation}\label{joha}
\left\{
\begin{array}{rcl}
^{\bf H}\mathfrak{D}_{T}^{\alpha ,\beta ;\psi }\left( \left\vert ^{\bf H}\mathfrak{D}_{0+}^{\alpha ,\beta ;\psi }\phi\left( x\right) \right\vert ^{p-2}\text{ }^{\bf H}\mathfrak{D}_{0+}^{\alpha ,\beta ;\psi } \phi(x) \right)   = & \lambda |\phi(x)|^{p-2}\phi(x)+b(x)|\phi(x)|^{q-1}\phi(x)  \\ 
I_{0+}^{\beta \left( \beta -1\right) ;\psi }\phi\left( 0\right)   = & 
I_{T}^{\beta \left( \beta -1\right) ;\psi }\phi\left( T\right) =0.
\end{array}%
\right. 
\end{equation}

Let $\theta=(\theta_{1},\theta_{2},...,\theta_{N})$, $T=(T_{1},T_{2},...,T_{N})$ and $\alpha=(\alpha_{1},\alpha_{2},...,\alpha_{N})$ where $0<\alpha_{1},\alpha_{2},...,\alpha_{N}<1$ with $\theta_{j}<T_{j}$, for all $j\in \left\{1,2,...,N \right\}$, $N\in\mathbb{N}$. Also put $\Lambda=I_{1}\times I_{2}\times \cdots \times I_{N}=[\theta_{1},T_{1}]\times [\theta_{2},T_{2}]\times\cdots \times [\theta_{N},T_{N}]$ where $T_{1},T_{2},...,T_{N}$ and $\theta_{1},\theta_{2},...,\theta_{N}$ positive constants. Consider also $\chi(\cdot)$ be an increasing and positive monotone function on $(\theta_{1},T_{1}),(\theta_{2},T_{2}),...,(\theta_{N},T_{N})$, having a continuous derivative $\chi'(\cdot)$ on $(\theta_{1},T_{1}],(\theta_{2},T_{2}],...,(\theta_{N},T_{N}]$. The $\chi$-Riemann-Liouville fractional partial integral of order $\alpha$ of $N$-variables $\phi=(\phi_{1},\phi_{2},...,\phi_{N})\in L^{1}(\Lambda)$ denoted by ${\bf I}^{\alpha,\chi}_{\theta,\xi_{j}} (\cdot)$, is defined by \cite{J3,J2,J1}
\begin{equation*}
    {\bf I}^{\alpha,\chi}_{\theta,\xi_{j}} \phi(\xi_{j})=\dfrac{1}{\Gamma(\alpha_{j})} \int \int \cdots \int_{\Lambda} \chi'(s_{j})(\chi(\xi_{j})- \chi(s_{j}))^{\alpha_{j}-1} \phi(s_{j}) ds_{j}
\end{equation*}
with $\chi'(s_{j})(\chi(\xi_{j})- \chi(s_{j}))^{\alpha_{j}-1}=\chi'(s_{1})(\chi(\xi_{1})- \chi(s_{1}))^{\alpha_{1}-1} \chi'(s_{2})(\chi(\xi_{2})- \chi(s_{2}))^{\alpha_{2}-1}$ $\cdots\chi'(s_{N})(\chi(\xi_{N})- \chi(s_{N}))^{\alpha_{N}-1}$ where  \\ $\Gamma(\alpha_{j})=\Gamma(\alpha_{1})\Gamma(\alpha_{2})\cdots\Gamma(\alpha_{N})$, $\phi(s_{j})=\phi(s_{1})\phi(s_{2})\cdots \phi(s_{N})$, $ds_{j}=ds_{1}ds_{2}\cdots ds_{N}$, for all $j\in \left\{1,2,...,N \right\}$. Analogously, it is defined ${\bf I}^{\alpha,\chi}_{T,\xi_{j}} (\cdot)$.

Let $\phi,\chi \in C^{n}(\Lambda)$ two functions such that $\chi$ is increasing and $\chi'(\xi_{j})\neq 0$ $j\in \left\{1,2,...,N \right\}$, $\xi_{j}\in \Lambda$. The $\chi$-Hilfer fractional partial derivative of $N$-variables, denoted by ${^{\mathbf H}\mathfrak{D}}^{\alpha,\beta;\chi}_{\theta,\xi_{j}}(\cdot)$, of order $\alpha$ and type $\beta$ $(0\leq \beta\leq 1)$, is defined by \cite{J3,J2,J1}
\begin{equation}\label{derivada}
{^{\mathbf H}\mathfrak{D}}^{\alpha,\beta;\chi}_{\theta,\xi_{j}}\phi(\xi_{j})= {\bf I}^{\beta(1-\alpha),\chi}_{\theta,\xi_{j}} \Bigg(\frac{1}{\chi'(\xi_{j})} \frac{\partial^{N}} {\partial \xi_{j}} \Bigg) {\bf I}^{(1-\beta)(1-\alpha),\chi}_{\theta,\xi_{j}} \phi(\xi_{j})
\end{equation}
with $\partial \xi_{j}=\partial \xi_{1}, \partial \xi_{2} \cdots \partial \xi_{N}$ and $\chi'(\xi_{j})=\chi'(\xi_{1}) \chi'(\xi_{2})\cdots\chi'(\xi_{N})$, for all $ j\in \left\{1,2,...,N \right\}$. Analogously it is defined ${^{\mathbf H}\mathfrak{D}}^{\alpha,\beta;\chi}_{T,\xi_{j}} (\cdot)$.

Throughout this work, we will use the following notations ${\bf I}^{\alpha,\chi}_{T} (\cdot):={\bf I}^{\alpha,\chi}_{T,\xi_{j}} (\cdot)$, ${\bf I}^{\alpha,\chi}_{\theta} (\cdot):={\bf I}^{\alpha,\chi}_{\theta,\xi_{j}} (\cdot)$, ${^{\mathbf H}\mathfrak{D}}^{\alpha,\beta;\chi}_{\theta} (\cdot):={^{\mathbf H}\mathfrak{D}}^{\alpha,\beta;\chi}_{\theta,\xi_{j}} (\cdot)$ and ${^{\mathbf H}\mathfrak{D}}^{\alpha,\beta;\chi}_{T} (\cdot):={^{\mathbf H}\mathfrak{D}}^{\alpha,\beta;\chi}_{T,\xi_{j}} (\cdot)$.

Motivated by the above works, in the present paper, we consider the fractional Dirichlet problem involving the $\gamma(\xi)$-Laplacian equation given by
\begin{eqnarray}\label{1}   
\begin{dcases}
\dT\left(\left|\dTo \phi(\xi)\right|^{\gamma(\xi)-2} \dTo \phi(\xi)\right)=\lambda \eta(\xi)|\phi|^{\omega(\xi)-2}\phi+
\mathcal{A}(\xi)|\phi|^{s(\xi)-2}\phi\\
\phi(\xi)=0 \quad {\rm on}\,\, \partial\Lambda
\end{dcases}
\end{eqnarray}
where $\gamma,\omega,s\in C(\overline{\Lambda})$ such that $1<\omega(\xi)<\gamma(\xi)<s(\xi)<\gamma_{\alpha}^{*}(\xi), \,\,\gamma_{\alpha}^{*}(\xi)=\frac{2\gamma(\xi)}{2-\alpha \gamma(\xi)}$ if $2>\alpha \gamma(\xi);\gamma_{\alpha}^{*}(\xi)=\infty$ if \\$2\leq \gamma(\xi),1<\gamma^{-}:=\displaystyle{ess\,inf_{\xi\in\Lambda} \gamma(\xi)\leq \gamma(\xi)\leq \gamma^{+}:=ess\,sup_{\xi\in\Lambda}\gamma(\xi)<\infty}$, $ 1<\omega^{-}\leq{\omega^{+}}<\gamma^{-}\leq \gamma^{+}<s^{-}\leq s^{+},\lambda>0\in\R$ and $\eta,\mathcal{A}\in C(\overline{\Lambda})$ are non-negative weight functions with compact support in $\Lambda:=[0,T]\times [0,T]$ and $\dTo(\cdot)$ and $\dT(\cdot)$ are $\chi$-Hilfer fractional derivative of order $1/\gamma(\xi)<\alpha<1$ and type $0\leq{\beta}\leq{1}$ given by Eq.(\ref{derivada}). The fractional operator $\dT\left(\Big|\dTo \phi(\xi)\Big|^{\gamma(\xi)-2} \dTo \phi(\xi)\right)$ is a generalization of the operator \\$\dT\left(\Big|\dTo \phi(\xi)\Big|^{\gamma-2} \dTo \phi(\xi)\right)$, in which $\gamma(\xi)=\gamma>1$.

The corresponding Euler functional of our problem (\ref{1}) is defined by
\begin{equation}\label{euler}
    \mathfrak{E}_\lambda(\phi)=\int_{\Lambda}\frac{1}{\gamma(\xi)}\Big|\dTo\,\phi\Big|^{\gamma(\xi)}d\xi-
\lambda\int_{\Lambda}\frac{1}{\omega(\xi)}\eta(\xi)|\phi|^{\omega(\xi)}d\xi-\int_{\Lambda}\frac{1}{s(\xi)}\mathcal{A}(\xi)|\phi|^{s(\xi)}d\xi.
\end{equation}

The main contributions and consequences of our paper, which becomes clearer in detail as follows:
\begin{enumerate}

\item  First, we present a new class of problems with $\gamma(\xi)$-Laplacian of variable exponents as detailed
by Eq.(\ref{1}).

\item  We prove some coercivity results and minimization of the Euler energy functional Eq.(\ref{euler}).

\item  We establish the multiplicity results of positive solutions for Eq.(\ref{1}) with nonnegative weight functions.

\item  We prove that the fractional Eq.(\ref{1}) has at least two positive solutions.

\item  A natural consequence of the results investigated here is the classic case when the limit $\alpha\rightarrow 1$.
\end{enumerate}

To investigate the main results as highlighted above, we make use of the Nehari manifold technique.

The rest of the article is divided as follows: Section 2, we present some important concepts and results for use throughout the paper, in particular, we highlight the proof of an extension to the Harnack inequality for the $\chi$-Hilfer fractional operator. Section 3, we investigate the main results of the paper, i.e, we discuss the existence and multiplicity of positive solutions to Eq.(\ref{1}) using the Nehari manifold and the Harnack inequality.

\section{Mathematical background - auxiliary results}

Consider the space \cite{Emunds1,Fan}
$$ \mathscr{L}^{\gamma(\xi)}(\Lambda)=\big\{\phi:\Lambda\rightarrow\mathbb{R}:\int_{\Lambda}|\phi(\xi)|^{\gamma(\xi)}d\xi<\infty\big\},$$
with the norm 
$${\lVert \phi \rVert}_{\gamma(\xi)}=\inf \bigg\{\delta>0:\int_{\Lambda}\bigg|\frac{\phi(\xi)}{\delta}\bigg|^{\gamma(\xi)}d\xi\leq{1}\bigg\}$$
(so-called Luxemburg norm) and $\Big(\mathscr{L}^{\gamma(\xi)}(\Lambda),{\lVert\,\cdot\, \rVert}_{\gamma(\xi)}\Big)$ is a Banach space. Write, $$\mathscr{L}^{\infty}=\big\{\gamma \in \mathscr{L}^{\infty}(\Lambda),\gamma^{-}>1\big\}.$$

Let $c(\xi)$ be a measurable real valued function and $\phi(\xi)>0$ for $\xi\in\Lambda$. Then the space $\mathscr{L}_{\phi(\xi)}^{\gamma(\xi)}(\Lambda)$ is defined with the norm \cite{Emunds1,Fan}
$${\lVert \phi \rVert}_{(\gamma(\xi),c(\xi))}=\inf\bigg\{\delta>0:\int_{\Lambda}c(\xi)\,\bigg|\frac{\phi(\xi)}{\delta}\bigg|^{\gamma(\xi)}d\xi\leq{1}\bigg\}.$$

\begin{definition}{\rm \cite{maria}} Let $0<\alpha\leq{1}, 0\leq{\beta}\leq{1}$ and $\gamma \in C^{+}(\overline{\Lambda})$. The left-sided $\chi$-fractional derivative space $\mathcal{H}_{\gamma(\xi),0}^{\alpha,\beta;\chi}:={\mathcal{H}}_{\gamma(\xi)}^{\alpha,\beta;\chi}(\Lambda)$
is defined by
\begin{eqnarray*}
\mathcal{H}_{\gamma(\xi),0}^{\alpha,\beta;\chi}&=&\left\{
\begin{array}{cc} \phi\in{L}^{\gamma(\xi)}(\Lambda):{^{\mathbf H}{\mathfrak{D}}}^{\alpha,\beta;\chi}_{\theta}\phi\in \mathscr{L}^{\gamma(\xi)}(\Lambda),
\phi(\Lambda)=0
\end{array}
\right\}.
\end{eqnarray*}
with the following norm
$$\|\phi\|_{\mathcal{H}_{\gamma(\xi)}^{\alpha\beta;\chi}}=\inf\left\{k>0:\int_{\Lambda}\bigg|\frac{\phi(\xi)}{k}\bigg|^{\gamma(\xi)}
+\bigg|\frac{{^{\mathbf H}{\mathfrak{D}}}^{\alpha,\beta;\chi}_{0+}\phi(\xi)}{k}\bigg|^{\gamma(\xi)}d\xi\leq{1}\right\}.$$
\end{definition}

The space $\mathcal{H}_{\gamma(\xi),0}^{\alpha,\beta;\chi}(\Lambda)$ is denoted by the closure of
$C_{0}^{\infty}(\Lambda)$ in $\mathcal{H}_{\gamma(\xi)}^{\alpha,\beta;\chi}(\Lambda)$. We will use
${\lVert \phi \rVert}_{\mathcal{H}_{\gamma(\xi)}^{\alpha,\beta;\chi}(\Lambda)}=\Big|\dTo\,\phi\Big|_{\gamma(\xi)}$
for $\phi\in\mathcal{H}_{\gamma(\xi)}^{\alpha,\beta;\chi}(\Lambda)$ in the following discussions.

\begin{proposition}{\rm\cite{maria,Sousa,Sousa10}} Let $0<\alpha \leq 1,$ $0\leq \beta \leq 1$ and $1<\gamma(\xi)<\infty $. Assume that $\alpha >1/\gamma(\xi)$ and the sequence $\left\{ \phi_{k}\right\} $ converges weakly to $\phi $ in $\mathcal{H}_{\gamma(\xi)}^{\alpha ,\beta ;\chi } (\Lambda;\mathbb{R})$ i.e., $\phi_{k}\rightharpoonup \phi$. Then $\phi_{k}\to \phi$ in $C\left( \Lambda,\mathbb{R}\right) $, i.e., $\left\Vert \phi-\phi_{k}\right\Vert _{\infty }\to 0$ as $k\rightarrow \infty$.
\end{proposition}

\begin{proposition}{\rm \cite{Sousa10,Sousa11}} The conjugate space of $\mathscr{L}^{\gamma(\xi)}(\Lambda)$ is $\mathscr{L}^{\gamma'(\xi)}(\Lambda)$, where $\frac{1}{\gamma'(\xi)}+\frac{1}{\gamma(\xi)}=1.$ For any $\phi\in \mathscr{L}^{\gamma(\xi)}(\Lambda)$ and $v\in \mathscr{L}^{\gamma'(\xi)}(\Lambda),$ we have
\begin{eqnarray*}
\bigg|\int_{\Lambda} \phi(\xi)v(\xi)d\xi\bigg|&\leq & \bigg(\frac{1}{\gamma^-}+\frac{1}{(\gamma')^-}\bigg){\lVert \phi \rVert}_{\gamma(\xi)}{\lVert v \rVert}_{\gamma'(\xi)}\\
&\leq &2{\lVert \phi \rVert}_{\gamma(\xi)}{\lVert v \rVert}_{\gamma'(\xi)}.
\end{eqnarray*}
\end{proposition}
\begin{proposition}{\rm \cite{Emunds1,Fan}\label{proposition 2-2}}
Denote $\rho(\phi)=\int_{\Lambda}|\phi(\xi)|^{\gamma(\xi)}d\xi,\,\,\forall \phi\in \mathscr{L}^{\gamma(\xi)}(\Lambda)$, then we have
\begin{enumerate}
\item ${\lVert \phi \rVert}_{\gamma(\xi)}<1\,\,(=1,>1)\Longleftrightarrow \rho(\phi)<1;$
\item ${\lVert \phi \rVert}_{\gamma(\xi)}>1\Rightarrow {\lVert \phi \rVert}_{\gamma(\xi)}^{\gamma^-}\leq \rho(\phi)\leq {\lVert \phi \rVert}_{\gamma(\xi)}^{\gamma^+};$
\item ${\lVert \phi \rVert}_{\gamma(\xi)}<1\Rightarrow {\lVert \phi \rVert}_{\gamma(\xi)}^{\gamma^-}\leq \rho(\phi)\leq {\lVert \phi \rVert}_{\gamma(\xi)}^{\gamma^+};$
\end{enumerate}
\end{proposition}
\begin{proposition} {\rm\cite{Emunds1,Fan}} If $\phi,\phi_n\in \mathscr{L}^{\gamma(\xi)}(\Lambda),n=1,2,\ldots$, then the follows statements are equivalent
to each other:
\begin{enumerate}
\item $\displaystyle \lim_{n\to\infty}{\lVert \phi_n-\phi\rVert}_{\gamma(\xi)}=0;$
\item $\displaystyle \lim_{n\to\infty}\rho(\phi_n-\phi)=0;$
\item $\phi_n\rightarrow \phi$ in measure on $\Lambda$ and $\displaystyle \lim_{n\to\infty}\rho(\phi_n)=\rho(\phi).$
\end{enumerate}
\end{proposition}

\begin{proposition} {\rm\cite{Emunds1,Fan,maria}} If $\gamma^{-}>1$ and $\gamma^{+}<\infty$, then the spaces  $\mathscr{L}^{\gamma(\xi)}(\Lambda)$,
$\mathscr{L}_{c(\xi)}^{\gamma(\xi)}(\Lambda)$ and $\mathcal{H}_{\gamma(\xi)}^{\alpha,\beta;\chi}(\Lambda)$ are
separable and reflexive Banach spaces.
\end{proposition}

\begin{proposition} {\rm\cite{Emunds1}} Let $\gamma(\xi)$ and $\omega(\xi)$ be measurable functions such that $\gamma(\xi)\in \mathscr{L}^{\infty}(\Lambda)$ and
$1\leq \gamma(\xi)\omega(\xi)\leq\infty$ for $\phi\in \xi\in\Lambda$. Let $\phi\in \mathscr{L}^{\omega(\xi)}(\Lambda),\,\,\phi\neq 0$. Then
\begin{eqnarray*}
&&|\phi|_{\gamma(\xi)\omega(\xi)}\leq{1} \quad \Rightarrow \quad |\phi|_{\gamma(\xi)\omega(\xi)}^{\gamma^+}\leq\big||\phi|^{\gamma(\xi)}\big|_{\omega(\xi)}\leq |\phi|_{\gamma(\xi)\omega(\xi)}^{\gamma^-}\\
&&|\phi|_{\gamma(\xi)\omega(\xi)}\geq{1} \quad \Rightarrow \quad |\phi|_{\gamma(\xi)\omega(\xi)}^{\gamma^-}\leq\big||\phi|^{\gamma(\xi)}\big|_{\omega(\xi)}\leq |\phi|_{\gamma(\xi)\omega(\xi)}^{\gamma^+}.
\end{eqnarray*}
\end{proposition}

Consider the following condition:

(${\bf A}_{1}$): Assume that the boundary of $\Lambda$ possesses the cone property {\rm \cite{Mashiyev}}.

\begin{theorem}{\rm \cite{Mashiyev}}\label{theorem 2-7}
Under the condition (${\bf A}_{1}$) and $\gamma \in C(\overline{\Lambda})$.
Suppose that \\$\mathcal{A}\in \mathscr{L}^{\gamma(\xi)}(\Lambda),\,\,\mathcal{A}(\xi)>0$ for $\xi\in\Lambda,\,\,\beta\in C(\overline{\Lambda})$
and $\beta^{-}>1,\,\,\beta_{0}^{-}\leq \beta_{0}(\xi)\leq \beta_{0}^{+}\,\,\left(\frac{1}{\beta(\xi)}+\frac{1}{\beta_0(\xi)}=1\right)$. If $h\in C(\overline{\Lambda})$ and
\begin{equation}\label{2.1}
    1<h(\xi)<\frac{\beta(\xi)-1}{\beta(\xi)}\gamma_{\alpha}^{*}, \quad \forall \xi\in\overline{\Lambda}
\end{equation}
or
$$1<\beta(\xi)<\frac{N\gamma(\xi)}{N\gamma(\xi)-h(\xi)(N-\gamma(\xi)}\cdot$$
then the embedding from $W^{1;\gamma(\xi)}(\Lambda)$ to $\mathscr{L}_{\mathcal{A}(\xi)}^{h(\xi)}(\Lambda)$ is compact.
Moreover, these is a constant $C_5>0$ such that the inequality
\begin{equation}\label{2.2}
\int_{\Lambda}\mathcal{A}(\xi)|\phi|^{h(\xi)}d\xi\leq{C_5}\Big({\lVert \phi\rVert}^{h^-}+{\lVert \phi\rVert}^{h^+}\Big)
\end{equation}
holds.
\end{theorem}
\begin{theorem}{\rm\cite{Mashiyev}}\label{theorem 2-8}
Under the condition (${\bf A}_{1}$) and $\gamma \in C(\overline{\Lambda})$.
Suppose that \\$\eta\in \mathscr{L}^{\alpha(\xi)}(\Lambda),\,\,\eta(\xi)>0$ for $\xi\in\Lambda,\,\,\alpha\in C(\overline{\Lambda})$ and $\alpha^{-}>1,\,\,\alpha_{0}^{-}\leq\alpha_{0}(\xi)\leq\alpha_{0}^{+}\,\,\left(\dfrac{1}{\alpha(\xi)}+\dfrac{1}{\alpha_0(\xi)}=1\right).$ If $\omega\in C(\overline{\Lambda}),\,\,
\gamma(\xi)<\dfrac{\alpha(\xi)}{\alpha(\xi)-1}\omega(\xi)$ and
$$1<\omega(\xi)<\frac{\alpha(\xi)-1}{\alpha(\xi)}\gamma_{\alpha}^{*}(\xi), \quad \forall \xi\in\overline{\Lambda}$$
or
$$\frac{N\gamma(\xi)}{N\gamma(\xi)-\omega(\xi)(N-\gamma(\xi))}<\alpha(\xi)<\frac{\gamma(\xi)}{\gamma(\xi)-\omega(\xi)},$$
then the embedding from $W^{1;\gamma(\xi)}(\Lambda)$ to $\mathscr{L}_{\eta(\xi)}^{\omega(\xi)}(\Lambda)$ is compact. Moreover,
there is a constant $C_7>0$ such that the inequality
$$\int_{\Lambda}\eta(\xi)|\phi|^{\omega(\xi)}d\xi\leq{C_7}\Big({\lVert \phi\rVert}^{\omega^-}+{\lVert \phi\rVert}^{\omega^+}\Big).$$
\end{theorem}
\begin{proposition}{\rm\cite{Mashiyev}}\label{proposition 2-9}
Assume that the conditions of \textnormal{Theorem \ref{theorem 2-7}} and \textnormal{Theorem \ref{theorem 2-8}} hold, respectively. Let $\phi\in W^{0;\gamma(\xi)}(\Lambda)$ then there are positive constants $C_8,C_9,C_{10},C_{11}>0$ such that the following inequalities hold
\begin{eqnarray}\label{2.4}
\int_{\Lambda}\mathcal{A}(\xi)|\phi|^{s(\xi)}d\xi\leq 
\begin{dcases}
C_8{\lVert \phi\rVert}^{s^+}, \quad \mbox{if}\,\,\,\lVert \phi\rVert>1,\\
C_9{\lVert \phi\rVert}^{s^-}, \quad \mbox{if}\,\,\,\lVert \phi\rVert<1,
\end{dcases}
\end{eqnarray}
and
\begin{eqnarray}
\int_{\Lambda}\eta(\xi)|\phi|^{\omega(\xi)}d\xi\leq 
\begin{dcases}
C_{10}{\lVert \phi\rVert}^{\omega^+}, \quad \mbox{if}\,\,\,\lVert \phi\rVert>1,\\
C_{11}{\lVert \phi\rVert}^{\omega^-}, \quad \mbox{if}\,\,\,\lVert \phi\rVert<1.
\end{dcases}
\end{eqnarray}
\end{proposition}

\begin{theorem}{\rm \cite{koooo}}\label{Harnack} {\rm(Harnack inequality)} Let $t_{\ast }\geq 0,0<\sigma _{1}<\sigma _{2}<\sigma _{3}$ and $\rho  >0$. Let further $\alpha \in (0,1)$, $0\leq \beta \leq 1$, $\chi(0)=0$ and $\phi_{0}\geq 0.$ Then for any function $\phi\in Z\left( t_{\ast },t_{\ast }+\sigma _{3}\rho \right) $ and that satisfies
\begin{equation}\label{equ1}
\partial _{t}^{\alpha ,\beta ;\chi }(\phi-\phi_{0})(t)=0,\text{ a.a.t}\in \left(t_{\ast },t_{\ast }+\sigma _{3}\rho \right)
\end{equation}
there holds the inequality
\begin{equation}\label{equ2}
\underset{W-}{\sup } \phi\leq \sigma _{3}\sigma _{1}\underset{W+}{\inf }\phi
\end{equation}
where $W-=\left( t_{\ast }+\sigma _{1}\rho ,t_{\ast }+\sigma _{2}\rho \right) $ and $W+=\left( t_{\ast }+\sigma _{2}\rho ,t_{\ast }+\sigma _{3}\rho \right) $.
\end{theorem}

\section{Existence and multiplicity of positive solutions}

Consider the Euler functional defined by Eq.(\ref{euler}). 
Then, by Theorem \ref{theorem 2-7} and Theorem \ref{theorem 2-8} and Proposition \ref{proposition 2-2}, yields
\begin{eqnarray*}
\mathfrak{E}_\lambda(\phi)&\geq &\frac{1}{\gamma^+}\int_{\Lambda}\Big|\dTo\,\phi\Big|^{\gamma(\xi)}d\xi-
\frac{\lambda}{\omega^-}\int_{\Lambda}\eta(\xi)|\phi|^{\omega(\xi)}d\xi-\frac{1}{s^-}\int_{\Lambda}\mathcal{A}(\xi)|\phi|^{hs\xi)}d\xi\\
&\geq &\frac{1}{\gamma^+}{\lVert \phi\rVert}^{\gamma^-}-\frac{\lambda}{\omega^-}C_7\Big({\lVert \phi\rVert}^{\omega^-}+{\lVert \phi\rVert}^{\omega^+}\Big)+\frac{1}{s^{-}}C_{5} \left( ||\phi||^{s^{-}}+||\phi||^{s^{+}}\right).
\end{eqnarray*}

Note that, $\mathfrak{E}_\lambda(\cdot)$ is not bounded below on
whole $\mathcal{H}_{\gamma(\xi)}^{\alpha,\beta;\chi}(\Lambda)$, since $\omega^+<\gamma^-\leq \gamma^+<s^-\leq s^+$, but must bounded on the Nehari manifold $\mathfrak{M}_{\lambda}(\Lambda)$ which is given by
$$\mathfrak{M}_{\lambda}(\Lambda)=\Big\{\phi\in\mathcal{H}_{\gamma(\xi)}^{\alpha,\beta;\chi}(\Lambda)\Big/ \{0\}:
\langle \mathfrak{E}'_\lambda(\phi),\phi\rangle=0\Big\}.$$
The all critical points of $\mathfrak{E}_\lambda$ must be on $\mathfrak{M}_\lambda(\Lambda)$
and local minimizes on $\mathfrak{E}_\lambda(\Lambda)$ are usually critical points of $\mathfrak{E}_\lambda$.
Thus, $\phi\in\mathfrak{M}_\lambda(\Lambda)$ if, and only if,
\begin{eqnarray}\label{3.1}
\mathbf{I}_\lambda(\phi)&:=&\langle \mathfrak{E}'_\lambda(\phi),\phi\rangle \notag\\&=& \int_{\Lambda}\Big|\dTo\,\phi\Big|^{\gamma(\xi)}d\xi-\lambda\int_{\Lambda}\eta(\xi)|\phi|^{\omega(\xi)}d\xi-\int_{\Lambda}\mathcal{A}(\xi)|\phi|^{s(\xi)}d\xi=0.\notag\\
\end{eqnarray}
Then, for $\phi\in\mathfrak{M}_\lambda(\Lambda)$, yields
\begin{eqnarray*}
&&\langle \mathbf{I}'_\lambda(\phi),\phi\rangle \notag\\&=&\int_{\Lambda}\gamma(\xi)\Big|\dTo\,\phi\Big|^{\gamma(\xi)}d\xi-\lambda\int_{\Lambda}\omega(\xi)\eta(\xi)|\phi|^{\omega(\xi)}d\xi-\int_{\Lambda}s(\xi)\mathcal{A}(\xi)|\phi|^{s(\xi)}d\xi\\
&\leq &(\gamma^+-\omega^-)\lambda\int_{\Lambda}\eta(\xi)|\phi|^{\omega(\xi)}d\xi-(\gamma^+-s^-)\int_{\Lambda}\mathcal{A}(\xi)|\phi|^{s(\xi)}d\xi.
\end{eqnarray*}
Now let's decompose the Nehari manifold $\mathfrak{M}_\lambda(\Lambda)$ into three parts
\begin{eqnarray*}
&&\mathfrak{M}^{+}_\lambda(\Lambda)=\big\{\phi\in\mathfrak{M}_\lambda(\Lambda):\langle \mathbf{I}'_\lambda(\phi),\phi\rangle>0\big\}\\
&&\mathfrak{M}^{-}_\lambda(\Lambda)=\big\{\phi\in\mathfrak{M}_\lambda(\Lambda):\langle \mathbf{I}'_\lambda(\phi),\phi\rangle<0\big\}\\
&&\mathfrak{M}^{0}_\lambda(\Lambda)=\big\{\phi\in\mathfrak{M}_\lambda(\Lambda):\langle \mathbf{I}'_\lambda(\phi),\phi\rangle=0\big\}.
\end{eqnarray*}
\begin{theorem}\label{theorem 3-1} Let $\phi_0$ be a local maximum or minimum for $\mathfrak{E}_\lambda$ on
$\mathfrak{M}_\lambda(\Lambda)$. If $\phi_0\notin \mathfrak{M}_\lambda^{0}(\Lambda)$, then
$\phi_0$ is a critical point of $\mathfrak{E}_\lambda$.
\end{theorem}
\begin{lemma}
The functional $\mathfrak{E}_\lambda$ is bounded and coercive below on $\mathfrak{M}_\lambda(\Lambda)$.
\end{lemma}
\begin{proof} Indeed, $\phi\in\mathfrak{M}_\lambda(\Lambda)$ and $\lVert \phi\rVert>1$. From (\ref{3.1}) and Proposition \ref{proposition 2-2} and Proposition \ref{proposition 2-9}, yields
\begin{eqnarray*}
&&\mathfrak{E}_\lambda(\phi)\notag\\&=&\int_{\Lambda}\frac{1}{\gamma(\xi)}\Big|\dTo\,\phi\Big|^{\gamma(\xi)}d\xi-
\lambda\int_{\Lambda}\frac{1}{\omega(\xi)}\eta(\xi)|\phi|^{\omega(\xi)}d\xi-
\int_{\Lambda}\frac{1}{s(\xi)}\mathcal{A}(\xi)|\phi|^{s(\xi)}d\xi\\
&\geq &\frac{1}{\gamma^+}\int_{\Lambda}\Big|\dTo\,\phi\Big|^{\gamma(\xi)}d\xi-\frac{\lambda}{\omega^-}\int_{\Lambda}\eta(\xi)|\phi|^{\omega(\xi)}d\xi\notag\\&-&\frac{1}{s^-}\left(\int_{\Lambda}\Big|\dTo\,\phi\Big|^{\gamma(\xi)}d\xi-\lambda\int_{\Lambda}\eta(\xi)|\phi|^{\omega(\xi)}d\xi\right)\\
&\geq &\left(\frac{1}{\gamma^+}-\frac{1}{s^-}\right)\int_{\Lambda}\Big|\dTo\,\phi\Big|^{\gamma(\xi)}d\xi+
\lambda\left(\frac{1}{s^-}-\frac{1}{\omega^-}\right)\int_{\Lambda}\eta(\xi)|\phi|^{\omega(\xi)}d\xi\\
&\geq &\left(\frac{s^- -\gamma^+}{{s^-}{\gamma^+}}\right){\|\phi\|}^{\gamma^-}-{C_{10}}\lambda\left(\frac{s^- -\omega^-}{{s^-}{\omega^-}}\right){\|\phi\|}^{\omega^+}.
\end{eqnarray*}
Since $\gamma^{-}>\omega^{+}$, so $\mathfrak{E}_\lambda(\phi)\rightarrow\infty$ as $\lVert \phi\rVert\rightarrow\infty$. Hence, $\mathfrak{E}_\lambda$ is bounded below and coercive on $\mathfrak{E}_\lambda(\Lambda)$.
\end{proof}
\begin{lemma}\label{Lemma 3-3}
There exists $\lambda_1>0$ such that for $0<\lambda<\lambda_1$ we have $\mathfrak{M}_\lambda^{0}=\emptyset$.
\end{lemma}
\begin{proof} Suppose otherwise, this is, $\mathfrak{M}_\lambda^{0}\neq\emptyset$ for all $\lambda\in\R\setminus\{0\}$.
Let $\phi\in\mathfrak{M}_\lambda^{0}(\Lambda)$ such that $\|\phi\|>1$. Then, using Eq.(\ref{3.1}), Eq.(\ref{2.4}) and
definition of $\mathfrak{M}_\lambda^{0}(\Lambda)$, yields
\begin{eqnarray*}
0&=&\langle \mathbf{I}'_\lambda(\phi),\phi\rangle\\
&=&\int_{\Lambda}\gamma(\xi)\Big|\dTo\,\phi\Big|^{\gamma(\xi)}d\xi-\lambda\int_{\Lambda}\omega(\xi)\eta(\xi)|\phi|^{\omega(\xi)}d\xi
-\int_{\Lambda}s(\xi)\mathcal{A}(\xi)|\phi|^{s(\xi)}d\xi\\
&\geq &\gamma^{-}\int_{\Lambda}\Big|\dTo\,\phi\Big|^{\gamma(\xi)}d\xi-\omega^{+}\left(\int_{\Lambda}\Big|\dTo\,\phi\Big|^{\gamma(\xi)}d\xi-\int_{\Lambda}\mathcal{A}(\xi)|\phi|^{s(\xi)}d\xi\right)\notag\\&-&s^{+}\int_{\Lambda}\mathcal{A}(\xi)|\phi|^{s(\xi)}d\xi\\
&\geq &(\gamma^{-}-\omega^{+})\int_{\Lambda}\Big|\dTo\,\phi\Big|^{\gamma(\xi)}d\xi+(\omega^{+}-s^{+})\int_{\Lambda}\mathcal{A}(\xi)|\phi|^{s(\xi)}d\xi.
\end{eqnarray*}
From Proposition \ref{proposition 2-9}, yields
\begin{eqnarray}
0\geq (\gamma^{-}-\omega^{+})\|\phi\|^{\gamma^-}+C_8(\omega^{+}-s^{+})\|\phi\|^{s^+}\nonumber\\
\|\phi\|\geq{C_{12}}\left(\frac{{\gamma^-}-\omega^+}{{s^+}-\omega^+}\right)^{\frac{1}{{s^+}-\gamma^{-}}}\cdot \label{3-2}
\end{eqnarray}
Similarly,
\begin{eqnarray*}
0&=&\langle \mathbf{I}'_\lambda(\phi),\phi\rangle\\
&=&{\gamma^+}\int_{\Lambda}\Big|\dTo\,\phi\Big|^{\gamma(\xi)}d\xi-\lambda{\omega^-}\int_{\Lambda}\eta(\xi)|\phi|^{\omega(\xi)}d\xi
-{s^-}\int_{\Lambda}\mathcal{A}(\xi)|\phi|^{s(\xi)}d\xi\\
&\leq &\gamma^{+}\int_{\Lambda}\Big|\dTo\,\phi\Big|^{\gamma(\xi)}d\xi-\lambda{\omega^{-}}\int_{\Lambda}\eta(\xi)|\phi|^{\omega(\xi)}d\xi\notag\\&-&s^{-}\left(\int_{\Lambda}\Big|\dTo\,\phi\Big|^{\gamma(\xi)}d\xi-\lambda\int_{\Lambda}\eta(\xi)|\phi|^{\omega(\xi)}d\xi\right).
\end{eqnarray*}
Using Proposition \ref{proposition 2-9}, yields
\begin{eqnarray}
0\leq (\gamma^{+}-s^{-})\|\phi\|^{\gamma^-}+\lambda{C_{10}}(s^{-}-\omega^{-})\|\phi\|^{\omega^+}\nonumber\\
\|\phi\|\geq{C_{13}}\left(\lambda\frac{{s^-}-\omega^-}{{s^-}-\gamma^+}\right)^{\frac{1}{{\gamma^-}-\omega^{+}}}\cdot \label{3-3}
\end{eqnarray}

If $\lambda$ is sufficiently small $\lambda=\left(\frac{s^{-}-\gamma^+}{s^{-}-\omega^{-}}\right)\left(\frac{\gamma^{-}-\omega^+}{s^+-\omega^+}\right)^{\frac{\gamma^{-}-\omega^+}{s^+ -\gamma^{-}}},$ then from inequalities (\ref{3-2}) and (\ref{3-3}) we get
$\|\phi\|<1$ is a contradiction. So $\mathfrak{M}_\lambda^{0}=\emptyset$.
\end{proof}
Using Lemma \ref{Lemma 3-3}, for $0<\lambda<\lambda_1$, we can write $\mathfrak{M}_\lambda(\Lambda)=\mathfrak{M}_\lambda^{+}(\Lambda)\cup \mathfrak{M}_\lambda^{-}(\Lambda)$. Then
$$\alpha_{\lambda}^{+}=\inf_{\phi\in\mathfrak{M}_\lambda^{+}(\Lambda)} \mathfrak{E}_\lambda(\phi) \quad
\mbox{and} \quad \alpha_{\lambda}^{-}=\inf_{\phi\in\mathfrak{M}_\lambda^{-}(\Lambda)}=\mathfrak{E}_\lambda(\phi).$$
\begin{lemma}\label{lemma 3-4}
If $0<\lambda<\lambda_1$, then for all $\phi\in\mathfrak{M}_\lambda^{+}(\Lambda),\,\,\mathfrak{E}_\lambda(\phi)<0.$
\end{lemma}
\begin{proof}
Indeed, consider $\phi\in\mathfrak{M}_\lambda^{+}(\Lambda)$. Using the definition of $\mathfrak{E}_\lambda(\phi)$, follows that
\begin{eqnarray}
\mathfrak{E}_\lambda(\phi)\leq\frac{1}{\gamma^-}\int_{\Lambda}\Big|\dTo\,\phi\Big|^{\gamma(\xi)}d\xi-\frac{\lambda}{\omega^+}\int_{\Lambda}\eta(\xi)|\phi|^{\omega(\xi)}d\xi-\frac{1}{s^+}\int_{\Lambda}\mathcal{A}(\xi)|\phi|^{s(\xi)}d\xi.\label{3-4}\notag\\
\end{eqnarray}
Since $\phi\in\mathfrak{M}_\lambda^{+}(\Lambda)$ and multiply (\ref{3.1}) by $(-\omega^-)$, yields
\begin{eqnarray}
\int_{\Lambda}\mathcal{A}(\xi)|\phi|^{h(\xi)}d\xi<\left(\frac{{\gamma^+}-\omega^-}{s^{-}-\omega^{-}}\right)\int_{\Lambda}\Big|\dTo\,\phi\Big|^{\gamma(\xi)}d\xi. \label{3-6}
\end{eqnarray}
Moreover, using (\ref{3.1}) together with the inequality (\ref{3-4}), one has
\begin{eqnarray}
\mathfrak{E}_\lambda(\phi)\leq\left(\frac{1}{\gamma^-}+\frac{1}{\omega^+}\right)\int_{\Lambda}\Big|\dTo\,\phi\Big|^{\gamma(\xi)}d\xi+\left(\frac{1}{\omega^+}-\frac{1}{s^+}\right)\int_{\Lambda}\mathcal{A}(\xi)|\phi|^{s(\xi)}d\xi .\label{3-7}
\end{eqnarray}
Applying the inequality (\ref{3-6}) in (\ref{3-7}), it follows
$$\mathfrak{E}_\lambda(\phi)<-\frac{({\gamma^-}-\omega^+)({s^+}-\gamma^-)}{{s^+}{\gamma^-}{\omega^+}}\|\phi\|_{\mathcal{H}_{\gamma(\xi)}^{\alpha,\beta;\chi}}^{\gamma}<0.$$
Hence, we have $\displaystyle \alpha_{\lambda}^{+}=\inf_{\phi\in\mathfrak{M}_\lambda^{+}(\Lambda)} \mathfrak{E}_\alpha(\phi)<0.$
\end{proof}
\begin{theorem}\label{theorem 3-5}
If $0<\lambda<\lambda_1,$ these exists a minimizer of $\mathfrak{E}_\lambda$ on $\mathfrak{M}_\lambda^{+}(\Lambda)$.
\end{theorem}
\begin{proof}
Since $\mathfrak{E}_\lambda(\cdot)$ is bounded below on $\mathfrak{M}_\lambda(\Lambda)$, so it is also about $\mathfrak{M}_\lambda^{+}(\Lambda)$. Then, there exists a minimizing sequence $\{\phi_n^{+}\}\subseteq\mathfrak{M}_\lambda^{+}(\Lambda)$ such that
$$\lim_{n\to\infty}\mathfrak{E}_\lambda(\phi_n^{+})=\inf_{\phi\in\mathfrak{M}_\lambda^{+}(\Lambda)}\mathfrak{E}_\lambda(\phi)=\alpha_{\lambda}^{+}<0.$$
Since $\mathfrak{E}_\lambda$ is coercive, $\phi_{n}^{+}$ is bounded in $\mathcal{H}_{\gamma(\xi)}^{\alpha,\beta;\chi}(\Lambda)$. Thus, we may assume that $\phi_{n}^{+}\rightharpoonup \phi_{0}^{+}\in\mathcal{H}_{\gamma(\xi)}^{\alpha,\beta;\chi}(\Lambda)$ and then we have
$$\phi_{n}^{+}\rightarrow \phi_{0}^{+} \,\,\,\mbox{in}\,\,\, \mathscr{L}_{\eta(\xi)}^{\omega(\xi)}(\Lambda)$$
and
$$\phi_{n}^{+}\rightarrow \phi_{0}^{+} \,\,\,\mbox{in}\,\,\, \mathscr{L}_{\mathcal{A}(\xi)}^{s(\xi)}(\Lambda).$$
Now, we shall prove $\phi_{n}^{+}\rightarrow \phi_{0}^{+}$ in $\mathcal{H}_{\gamma(\xi)}^{\alpha,\beta;\chi}(\Lambda)$.
Otherwise, suppose $\phi_{n}^{+}\not\rightarrow \phi_{0}^{+}$ in $\mathcal{H}_{\gamma(\xi)}^{\alpha,\beta;\chi}(\Lambda)$.
Then,
$$\int_{\Lambda}\Big|\dTo\,\phi_{0}^{+}\Big|^{\gamma(\xi)}d\xi<\lim_{n\to \infty}\inf \int_{\Lambda}
\Big|\dTo\,\phi_{n}^{+}\Big|^{\gamma(\xi)}d\xi.$$
Moreover, by the compact embeddings, yields
\begin{eqnarray*}
\int_{\Lambda}\eta(\xi)|\phi_{0}^{+}|^{\omega(\xi)}d\xi&=&\lim_{n\to \infty}\inf \int_{\Lambda}
\eta(\xi)|\phi_{n}^{+}|^{\omega(\xi)}d\xi\\
\int_{\Lambda}\mathcal{A}(\xi)|\phi_{0}^{+}|^{s(\xi)}d\xi&=&\lim_{n\to \infty}\inf \int_{\Lambda}
\mathcal{A}(\xi)|\phi_{n}^{+}|^{s(\xi)}d\xi.
\end{eqnarray*}
Using $\langle \mathfrak{E}'_\lambda(\phi_{n}^{+}),\phi_{n}^{+}\rangle=0$
and Theorem \ref{theorem 2-8}, we obtain
\begin{eqnarray*}
\mathfrak{E}_\lambda(\phi_{n}^{+})&\geq &\left(\frac{1}{\gamma^+}-\frac{1}{s^-}\right)\int_{\Lambda}\Big|\dTo\,\phi_{n}^{+}\Big|^{\gamma(\xi)}d\xi+\lambda\left(\frac{1}{s^-}-\frac{1}{\omega^-}\right)\int_{\Lambda}\eta(\xi)|\phi_{n}^{+}|^{\omega(\xi)}d\xi,
\end{eqnarray*}
\begin{eqnarray*}
&&\lim_{n\to \infty}\mathfrak{E}_\lambda(\phi_{n}^{+})\notag \\ &\geq &\left(\frac{1}{\gamma^+}-\frac{1}{s^-}\right)\lim_{n\to \infty}\int_{\Lambda}\Big|\dTo\,\phi_{n}^{+}\Big|^{\gamma(\xi)}d\xi+\lambda\left(\frac{1}{s^-}-\frac{1}{\omega^-}\right)\lim_{n\to \infty}\int_{\Lambda}\eta(\xi)|\phi_{n}^{+}|^{\omega(\xi)}d\xi
\end{eqnarray*}
and
\begin{eqnarray*}
\alpha_{\lambda}^{+}&=&\inf_{\phi\in \mathfrak{M}_{\lambda}^{+}}\mathfrak{E}_{\lambda}(\phi)>\left(\frac{1}{\gamma^+}-\frac{1}{s^-}\right)\|\phi_{0}^{+}\|^{\gamma^-}+{C_7}\lambda\left(\frac{1}{s^-}-\frac{1}{\omega^-}\right)\Big(\|\phi_{0}^{+}\|^{\omega^-}+\|\phi_{0}^{+}\|^{\omega^+}\Big),
\end{eqnarray*}
since $\gamma^{-}>\omega^{+}$, for $\|\phi_0^{+}\|>1$, yields
$$\alpha_{\lambda}^{+}=\inf_{\phi\in \mathfrak{M}_{\lambda}^{+}}\mathfrak{E}_{\lambda}(\phi)>0.$$

So, $\phi\in\mathfrak{M}_{\lambda}^{+}(\Lambda)$ (see Lemma \ref{lemma 3-4}), one has $\mathfrak{E}_{\lambda}(\phi)<0$. So this is a contradiction. Hence, $\phi\in\mathfrak{M}_{\lambda}^{+}(\Lambda)$
in $\mathcal{H}_{\gamma(\xi)}^{\alpha,\beta;\chi}(\Lambda)$ and
$$\mathfrak{E}_{\lambda}(\phi_{0}^{+})=\lim_{n\to\infty}\mathfrak{E}_{\lambda}(\phi_{n}^{+})=
\inf_{\phi\in \mathfrak{M}_{\lambda}^{+}}\mathfrak{E}_{\lambda}(\phi).$$
Thus, $\phi_{0}^{+}$ is a minimizer for $\mathfrak{E}_{\lambda}$ on $\mathfrak{M}_{\lambda}^{+}(\Lambda).$
\end{proof}
\begin{lemma}
If $0<\lambda<\lambda_1$, then for all $\phi\in\mathfrak{M}_{\lambda}^{-}(\Lambda),$ $\mathfrak{E}_{\lambda}(\phi)>0.$
\end{lemma}
\begin{proof}
Consider $\phi\in\mathfrak{M}_{\lambda}(\Lambda).$ Using the definition of $\mathfrak{E}_{\lambda}(\Lambda)$ and (\ref{3.1}),
yields
\begin{eqnarray}\label{I}
\mathfrak{E}_{\lambda}(\phi)\geq \frac{1}{\gamma^+}\int_{\Lambda}\left(\Big|\dTo\,\phi\Big|^{\gamma(\xi)}
-\frac{\lambda}{\omega^-}\eta(\xi)|\phi|^{\omega(\xi)}\right)d\xi-\frac{1}{s^-}\int_{\Lambda}\mathcal{A}(\xi)|\phi|^{s(\xi)}d\xi \notag\\
\end{eqnarray}
and
\begin{eqnarray}
\int_{\Lambda}\mathcal{A}(\xi)|\phi|^{s(\xi)}d\xi=\int_{\Lambda}\Big|\dTo\,\phi\Big|^{\gamma(\xi)}d\xi-\lambda\int_{\Lambda}\eta(\xi)|\phi|^{\omega(\xi)}d\xi. \label{II}
\end{eqnarray}
Using Eq.(\ref{I})-Eq.(\ref{II}), Propositions \ref{proposition 2-2} and {\ref{proposition 2-9}} and
the condition $\gamma^{-}>\omega^{+}$, yields
\begin{eqnarray*}
\mathfrak{E}_{\lambda}(\phi)&\geq & \frac{1}{\gamma^+}\int_{\Lambda}\Big|\dTo\,\phi\Big|^{\gamma(\xi)}d\xi\notag\\&-&\frac{\lambda}{\omega^-}\int_{\Lambda}\eta(\xi)|\phi|^{\omega(\xi)}d\xi-\frac{1}{s^-}\left(\int_{\Lambda}\Big|\dTo\,\phi\Big|^{\gamma(\xi)}d\xi-\lambda\int_{\Lambda}\eta(\xi)|\phi|^{\omega(\xi)}d\xi\right)\\
&\geq &\left(\frac{1}{\gamma^+}-\frac{1}{s^-}\right)\int_{\Lambda}\Big|\dTo\,\phi\Big|^{\gamma(\xi)}d\xi+\lambda\left(\frac{1}{s^-}-\frac{1}{\omega^-}\right)\int_{\Lambda}\eta(\xi)|\phi|^{\omega(\xi)}d\xi\\
&\geq &\left(\frac{1}{\gamma^+}-\frac{1}{s^-}\right)\|\phi\|^{\gamma^-}+{C_{10}}\lambda\left(\frac{1}{s^-}-\frac{1}{\omega^-}\right)\|\phi\|^{\omega^+}\\
&\geq &\left(\frac{{s^-}-{\gamma^+}}{{\gamma^+}{s^-}}+C_{10}\frac{{\omega^-}-{s^-}}{{s^-}{\omega^-}}\right)\|\phi\|^{\gamma^-}.
\end{eqnarray*}
So, if we choose $\lambda<\frac{{\omega^-}({s^-}-{\gamma^+})}{C_{10}\gamma^{+}({s^-}-{\omega^-})},$ we get
$\mathfrak{E}_{\lambda}(\phi)>0$. Consider  $\mathfrak{M}_{\lambda}(\Lambda)=
\mathfrak{M}_{\lambda}^{+}(\Lambda)\cup\mathfrak{M}_{\lambda}^{-}(\Lambda)$ (see Lemma \ref{Lemma 3-3}),
$\mathfrak{M}_{\lambda}^{+}(\Lambda)\cap \mathfrak{M}_{\lambda}^{-}(\Lambda)=\emptyset$, and
Lemma \ref{lemma 3-4}, one has $\phi\in \mathfrak{M}_{\lambda}^{-}(\Lambda).$
\end{proof}
\begin{theorem}\label{theorem 3-7}
If $0<\lambda<\lambda_1$, there exists a minimizer of $\mathfrak{E}_{\lambda}(\cdot)$ on $\mathfrak{M}_{\lambda}^{-}(\Lambda).$
\end{theorem}
\begin{proof}
Since $\mathfrak{E}_{\lambda}$ is bounded below on $\mathfrak{M}_{\lambda}(\Lambda)$ and
so on $\mathfrak{M}_{\lambda}^{-}(\Lambda)$, then there exists a minimizing sequence
$\{\phi_{n}^{-}\}\subseteq\mathfrak{M}_{\lambda}^{-}(\Lambda)$ such that
$$\lim_{n\to\infty}\mathfrak{E}_{\lambda}(\phi_{n}^{-})=\inf_{\phi\in\mathfrak{M}_{\lambda}^{-}(\Lambda)}\mathfrak{E}_{\lambda}(\Lambda)=\alpha_{\lambda}^{-}>0.$$
Since $\mathfrak{E}_{\lambda}$ is coercive, $\phi_{n}^{-}$ is bounded in $\mathcal{H}_{\gamma(\xi)}^{\alpha,\beta;\chi}(\Lambda)$, we can may assume that $\phi_n^{-}\rightharpoonup \phi_{0}^{-}$ in $\mathcal{H}_{\gamma(\xi)}^{\alpha,\beta;\chi}(\Lambda)$. Using the compact embeddings,  follows that
$$\phi_{n}^{-}\rightarrow \phi_{0}^{-} \,\,\,\mbox{in}\,\,\, \mathscr{L}_{\eta(\xi),\chi}^{\omega(\xi)}(\Lambda)$$
and
$$\phi_{n}^{-}\rightarrow \phi_{0}^{-} \,\,\,\mbox{in}\,\,\, \mathscr{L}_{\mathcal{A}(\xi),\chi}^{s(\xi)}(\Lambda).$$
Moreover, if $\phi_{0}^{-}\in\mathfrak{M}_{\lambda}^{-}(\Lambda)$, then there is a constant $t>0$ such that
$t{\phi}_{0}^{-}\in\mathfrak{M}_{\lambda}^{-}(\Lambda)$ and $\mathfrak{E}_{\lambda}(\phi_0^{-})\geq\mathfrak{E}_{\lambda}(t\phi_{0}^{-})$. Indeed, since
\begin{eqnarray*}
&&\left< \mathbf{I}'_{\lambda}(\phi), \phi\right>\notag\\&=& \int_{\Lambda}\gamma(\xi)\Big|\dTo\,\phi\Big|^{\gamma(\xi)}d\xi-\lambda\int_{\Lambda}\omega(\xi)\eta(\xi)|\phi|^{\omega(\xi)}d\xi-\int_{\Lambda}s(\xi)\mathcal{A}(\xi)|\phi|^{s(\xi)}d\xi,
\end{eqnarray*}
then
\begin{eqnarray*}
&&\left<\mathbf{I}'_{\lambda}(t\phi_{0}^{-}), t\phi_{0}^{-} \right>\notag\\&=&\int_{\Lambda}\gamma(\xi)\Big|\dTo\,(t\phi_{0}^{-})\Big|^{\gamma(\xi)}d\xi-\lambda\int_{\Lambda}\omega(\xi)\eta(\xi)|t\phi_{0}^{-}|^{\omega(\xi)}d\xi\notag\\&-&\int_{\Lambda}s(\xi)\mathcal{A}(\xi)|t\phi_{0}^{-}|^{s(\xi)}
d\xi\\
&\geq &{t^{\gamma^+}}{\gamma^+}\int_{\Lambda}\Big|\dTo\,\phi_{0}^{-}\Big|^{\gamma(\xi)}d\xi-\lambda {t^{\omega^-}}{\omega^-}\int_{\Lambda}\eta(\xi)|\phi_{0}^{-}|^{\omega(\xi)}d\xi\notag\\&-&{t^{s^-}}{s^-}\int_{\Lambda}\mathcal{A}(\xi)|\phi_{0}^{-}|^{s(\xi)}d\xi.
\end{eqnarray*}
Note that $\mathbf{I}'_{\lambda}(t\phi_{0}^{-})<0$, since $\omega^-<\gamma^+<s^-$, and under the assumptions on $a$ and $\mathcal{A}$. So using the definition of $\mathfrak{M}_{\lambda}^{-}(\Lambda)$, follows that $t\phi_{0}^{-}\in\mathfrak{M}_{\lambda}^{-}(\Lambda)$.

{\bf Affirmation: $\phi_n^{-}\rightarrow \phi_{0}^{-}$ in $\mathcal{H}_{\gamma(\xi)}^{\alpha,\beta;\chi}(\Lambda)$ }

Then using the fact that
$$\int_{\Lambda}\Big|\dTo\,\phi_{0}^{-}\Big|^{\gamma(\xi)}d\xi<\lim_{n\to\infty}\inf\int_{\Lambda}\Big|\dTo\,\phi_{n}^{-}\Big|^{\gamma(\xi)}d\xi,$$
yields
\begin{align*}
&\mathfrak{E}_{\lambda}(t\phi_{0}^{-})\notag\\&\leq \frac{t^{\gamma^+}}{\gamma^-}\int_{\Lambda}\Big|\dTo\,\phi_{0}^{-}\Big|^{\gamma(\xi)}d\xi-\lambda\frac{t^{\omega^-}}{\omega^+}\int_{\Lambda}\eta(\xi)|\phi_{0}^{-}|^{\omega(\xi)}d\xi-
\frac{t^{s^-}}{s^-}\int_{\Lambda}\mathcal{A}(\xi)|\phi_{0}^{-}|^{s(\xi)}d\xi\\
&<\lim_{n\to \infty}\left[\frac{t^{\gamma^+}}{\gamma^-}\int_{\Lambda}\Big|\dTo\,\phi_{n}^{-}\Big|^{\gamma(\xi)}d\xi-\lambda\frac{t^{\omega^-}}{\omega^+}\int_{\Lambda}\eta(\xi)|\phi_{n}^{-}|^{\omega(\xi)}d\xi\right.\\
 &\left.-
\frac{t^{s^-}}{s^-}\int_{\Lambda}\mathcal{A}(\xi)|\phi_{n}^{-}|^{s(\xi)}d\xi\right]\notag\\
&\leq \lim_{n\to\infty}\mathfrak{E}_{\lambda}(t\phi_{n}^{-})\leq\lim_{n\to\infty}\mathfrak{E}_{\lambda}(\phi_{n}^{-})=\inf_{\phi\in\mathfrak{M}_{\lambda}^{-}(\Lambda)}\mathfrak{E}_{\lambda}(\phi)=\alpha_{\lambda}^{-}.
\end{align*}
This implies that $\displaystyle \mathfrak{E}_{\lambda}(t\phi_{0}^{-})<\inf_{\phi\in\mathfrak{M}_{\lambda}^{-}(\Lambda)}\mathfrak{E}_{\lambda}(\phi)=\alpha_{\lambda}^{-},$ which is a contradiction. Hence, $\phi_{n}^{-}\rightarrow \phi_{0}^{-}$ in $\mathcal{H}_{\gamma(\xi)}^{\alpha,\beta;\chi}(\Lambda)$ and so 
$$\mathfrak{E}_{\lambda}(\phi_{0}^{-})=\lim_{n\to\infty}\mathfrak{E}_{\lambda}(\phi_{n}^{-})=
\inf_{\phi\in\mathfrak{M}_{\lambda}^{-}(\Lambda)}\mathfrak{E}_{\lambda}(\phi).$$
Thus, $\phi_{0}^{-}$ is a minimizer for $\mathfrak{E}_{\lambda}$ on $\mathfrak{M}_{\lambda}^{-}(\Lambda)$.
\end{proof}
\begin{corollary} Using \rm{Theorem \ref{theorem 3-5}} and {\rm Theorem \ref{theorem 3-7}}, there exists $\phi_{0}^{+}\in\mathfrak{M}_{\lambda}^{+}(\Lambda)$ and
$\phi_{0}^{-}\in\mathfrak{M}_{\lambda}^{-}(\Lambda)$ such that $\displaystyle\mathfrak{E}_{\lambda}(\phi_{0}^{+})=\inf_{\phi\in\mathfrak{M}_{\lambda}^{+}(\Lambda)}\mathfrak{E}_{\lambda}(\phi)$ and
$\displaystyle\mathfrak{E}_{\lambda}(\phi_{0}^{-})=\inf_{\phi\in\mathfrak{M}_{\lambda}^{-}(\Lambda)}\mathfrak{E}_{\lambda}(\phi)$. Moreover, since $\mathfrak{E}_{\lambda}\big(\phi_{0}^{\pm}\big)=\mathfrak{E}_{\lambda}\big(|\phi_{0}^{\pm}|\big)$ and $|\phi_{0}^{\pm}|\in\mathfrak{M}_{\lambda}^{\pm}(\Lambda)$, we may
assume $\phi_{0}^{\pm}\geq{0}.$ Now making use \rm{Theorem \ref{theorem 3-1}}, $\phi_{0}^{\pm}$ are critical points of $\mathfrak{E}_{\lambda}$ on $\mathcal{H}_{\gamma(\xi)}^{\alpha,\beta;\chi}(\Lambda)$
and hence are weak solutions of \rm{(\ref{1})}. Finally, using the Harnack inequality \rm{(Theorem \ref{Harnack})}, we concluded that $\phi_0^{\pm}$ are positive solutions of \rm{(\ref{1})}.
\end{corollary}

\section{Conclusion and remarks}

We end this paper with the objectives achieved, that is, we investigate the existence and multiplicity of the Dirichlet fractional problem involving the equation $\gamma(\xi)$-Laplacian with non-negative weight functions using some variational techniques and the Nehari manifold. The particular choice of the $\psi$-Hilfer operator to work with this problem is motivated by several factors, in particular, the wide range of possible particular cases from the choice of the $\psi$ function. In addition, we can highlight the variational structure created through the $\psi$-Hilfer fractional operator, which makes the results more attractive. On the other hand, there is the memory factor that is directly linked to fractional operators. However, there are some open problems and future work motivated by Dirichlet problems proposed via the $\psi$-Hilfer fractional operator, as shown below:

1. Note that the results investigated here were considered only when $\frac{1}{\gamma(\xi)}<\alpha<1$, because for $0<\alpha<\frac{1}{\gamma( \xi)}$ we need density results, and we don't have them yet. This is an open problem in the area.

2. We can think of working the problem (\ref{1}) in the context of a double phase or involving a Kirchhoff-type equation.

The results presented above contribute significantly to the area of fractional operators with $p(x)$-Laplacian equations and will certainly serve as a basis and motivation for other future works, in particular, for the problems highlighted above.

\vspace{3.5em}
{\bf Acknowledgements.} All authors’ contributions to this manuscript are the same. All authors read and approved the final manuscript. We are very grateful to the anonymous reviewers for their useful comments that led to improvement of the manuscript.
\vspace{1.0em}

{\bf Data Availability Statement.} Not applicable.
\vspace{1.0em}

{\bf Conflicts of Interest.} The authors declare no conflict of interest.

\vspace{1cc}

\noindent
{\bf J. Vanterler da C. Sousa}\\
Center for Mathematics, Computing and Cognition,\\ Federal University of ABC, \\Avenida dos Estados, 5001, Bairro Bangu, 09.210-580, Santo Andre, SP - Brazil\\
E-mail: {\it jose.vanterler@ufabc.edu.br,vanterlermatematico@hotmail.com}

\vspace{0.1cc}

\noindent
{\bf D. S. Oliveira}\\
Universidade Tecnol\'ogica Federal do Paran\'a,\\ 85053-525, Guarapuava-PR, Brazil\\
E-mail: {\it oliveiradaniela@utfpr.edu.br}

\vspace{0.1cc}

\noindent
{\bf Ravi P. Agarwal}\\
Department of Mathematics, \\Texas A\&M University-Kingsville, Kingsville, TX 78363, USA\\
E-mail: {\it agarwal@tamuk.edu}

\end{document}